\documentclass{amsart}
\usepackage{a4wide,amssymb}
\usepackage{xcolor}
\usepackage[normalem]{ulem}
\usepackage{enumitem}

\newtheorem{theorem}{Theorem}
\newtheorem{corollary}{Corollary}
\newtheorem{example}{Example}
\newtheorem{problem}{Problem}
\newtheorem{lemma}{Lemma}
\newtheorem{proposition}{Proposition}
\newtheorem{claim}{Claim}
\theoremstyle{definition}
\newtheorem{definition}{Definition}
\newtheorem{remark}{Remark}

\newcommand{\e}{\varepsilon}
\newcommand{\IR}{\mathbb R}
\newcommand{\IZ}{\mathbb Z}
\newcommand{\IQ}{\mathbb Q}
\newcommand{\IN}{\mathbb N}
\newcommand{\w}{\omega}
\newcommand{\defeq}{:=}

\title{Plastic metric spaces and groups}
\author{Taras Banakh, Oles Mazurenko, and Olesia Zavarzina}
\email{t.o.banakh@gmail.com, oles.mazurenko@lnu.edu.ua, olesia.zavarzina@yahoo.com}
\address{T.~Banakh, O.~Mazurenko: Ivan Franko National University of Lviv, Ukraine}
\address{O.~Zavarzina: V.N.~Karazin Kharkiv National University, Kharkiv, Ukraine}

\begin{document}
	
\begin{abstract}
 A metric space is {\em plastic} if all its non-expansive bijections are isometries. We prove three main results: (1) every countable dense subspace of a normed space is not plastic, (2) every $k$-crowded separable metric space contains a plastic dense subspace, and (3) every strictly convex separable metric group contains a plastic dense subgroup.
\end{abstract}
\maketitle

\section{Introduction}

This paper is devoted to plastic metric spaces and groups. A metric space $(X,d)$ is called {\em plastic} if every non-expansive bijection of $X$ is non-contractive. A map $f:X\to X$  is  {\em non-expansive} (resp. {\em non-contractive}) if $d(f(x),f(y))\le d(x,y)$ (resp. $d(f(x),f(y))\ge d(x,y)$) for all $x,y\in X$. If $f$ is both non-expansive and non-contractive, then $f$ is called an {\em isometry}. A metric group is plastic if its underlying metric space is plastic. An simple example of a non-plastic metric group is the real line $\mathbb R$ or, more generally, any normed space over the field of real numbers. In some references (e.g., \cite{KZ}, \cite{LangZav}, \cite{Leo}, \cite{Naimpally}, \cite{Zav}), plastic metric spaces are called {\em expand-contract plastic}.
\smallskip

	 
	

The first ideas of this paper appeared during the problem session at the V International conference dedicated to the 145th anniversary of Hans Hahn which took place in September, 2024, in Chernivtsi, Ukraine. The authors discussed unsolved problems on the plasticity of subsets of the real line. In spite of simplicity of formulations, problems related to plasticity are often difficult to deal with. For instance, no simple characterization of plastic subsets of the reals is known, not to mention general metric spaces. Some initial results concerning plastic subspaces of reals were obtained in \cite{LangZav}. Regarding general metric spaces, it is known that every totally bounded metric space is plastic, even is a stronger sense, see \cite{Naimpally}. Roughly speaking, the strong plasticity of a metric space means that increase of a distance between some pair of points under action of any bijection is necessarily compensated by decrease  of the distance between some other pair of points of this space. More details and results about strong plasticity and so-called uniform plasticity for bijections of coinciding or  different metric spaces can be found in \cite{KZ}.

There is a number of publications devoted to plasticity of the unit balls of Banach spaces and to extension of this problem: for which Banach spaces $X$ and $Y$ any non-expansive bijection between its unit balls is an isometry (see \cite{AnKaZa},  \cite{KZ2}, \cite{HLZ}, \cite{Leo}, \cite{Fakhoury}, \cite{CKOW2016}, \cite{Zav}). For some Banach spaces it was possible to prove that a non-expansive bijection between the unit balls is an isometry only some  additional conditions imposed on the mapping (which may be called conditional plasticity) \cite{Leo}, \cite{Kurik}. However, the problem of plasticity for the unit ball of an arbitrary Banach space is still open as well as its extension to the case of two different spaces. At least there is no examples of Banach spaces with non-plastic unit ball. On the other hand, there are plastic and non-plastic ellipsoids in Hilbert spaces. The full classification of so called linearly plastic ellipsoids in Hilbert spaces was given in \cite{Zav2} in terms of semi-axes of an ellipsoid and in \cite{KarZav} in terms of spectrum of the self-adjoint operator, corresponding to the ellipsoid.
\smallskip

This paper contains three principal results. The first of them establishes non-plasticity of countable dense subspaces in normed spaces. A subspace $X$ of a metric space $Y$ is called {\em plastic} if $X$ is plastic with respect to the metric inherited from $Y$. A set $X$ in a space $Y$ is {\em dense} if it intersects every nonempty open subset of $Y$.

\begin{theorem}\label{t:main1} Every countable dense subspace of any normed space is not plastic.
\end{theorem}

On the other hand, we have the following example, first constructed by Wojciech Bielas \cite{Bielas}.

\begin{example}\label{e:Bielas} The real line contains a plastic dense subspace $X$ of cardinality $\aleph_1$ such that every non-expansive bijection of $X$ is the identity map of $X$.
\end{example}

Theorem~\ref{t:main1} and Example~\ref{e:Bielas} imply the following characterization, first observed by Bielas \cite{Bielas}.

\begin{corollary}[Bielas]\label{c:Bielas-CH}The Continuum Hypothesis holds if and only if every dense subspace $X\subseteq\IR$ of cardinality $|X|<|\IR|$ is not plastic.
\end{corollary}

All normed spaces considered in this paper are over the field of real numbers.

The countability is essential in Theorem~\ref{t:main1} because of the following theorem that establishes the existence of plastic dense subspaces in every $k$-crowded separable metric space. A metric space $X$ is {\em $k$-crowded} if each nonempty open subset of $X$ contains an uncountable compact set. 

\begin{theorem}\label{t:main2} Every $k$-crowded separable metric space  contains a plastic dense subspace $X$ such that every non-expansive bijection of $X$ is the identity map of $X$.
\end{theorem}

Our last principal result treats plastic metric abelian groups. A {\em metric group} is a group endowed with a translation-invariant metric. In each metric group, the inversion and all translations are isometries, so metric groups necessarily have many isometries. More precisely, for any elements $a,b$ of a metric group $G$ and any sign $s \in\{-1,1\}$, the function $G\to G$, $x\mapsto ax^sb$, is an isometry of $G$. We define a metric group $G$ to be {\em plastically rigid} if each non-expansive bijection $f$ of $G$ is of the form $f(x)=ax^sb$ for some $a,b\in G$ and $s\in\{-1,1\}$. It is clear that each plastically rigid metric group is plastic (as a metric space).

Our third principal theorem establishes the existence of a plastically rigid dense subgroup in every strictly convex metric abelian group. A metric space $(X,{\sf d})$ is called {\em strictly convex} if for every points $x,y\in X$ and positive real numbers $a,b$ with $a+b={\sf d}(x,y)$, there exists a unique point $z\in X$ such that ${\sf d}(x,z)=a$ and ${\sf d}(z,y)=b$. The strict convexity plays an important role in Banach space geometry, optimization, approximation theory, and fixed point theory (cf. \cite{Barbu}, \cite{BorLewis}, \cite{GK}, \cite{Ist}). By  \cite{BM}, every strictly convex metric abelian group has the canonical structure of a normed space.

\begin{theorem}\label{t:main3} Every strictly convex separable metric group contains a plastically rigid dense subgroup.
\end{theorem}

Theorems~\ref{t:main1}, \ref{t:main2}, \ref{t:main3} will be proved in Sections~\ref{sec1}, \ref{sec2}, \ref{sec3}, respectively. In Section~\ref{s:Cantor} we recall a classical result on Cantor sets in analytic spaces, and in Section~\ref{sec2a} we prove some auxiliary results that are used in the proof of Theorem~\ref{t:main3} (which is long and difficult). In the final Section~\ref{s:openprob} we collect some additional remarks and open problems related to plasticity.



\section{Proof of Theorem~\ref{t:main1}}\label{sec1}

We shall deduce Theorem~\ref{t:main1} from suitable known results on Lipschitz countable dense homogeneous metric spaces.

First we recall the notion of an $\e$-isometry, where $\e$ is a positive real number.

\begin{definition} A  function $f:X\to Y$ between two metric spaces $(X,d_X)$ and $(Y,d_Y)$ is called an {\em $\e$-isometry} if
$$(1-\e){\cdot}d_X(x,y)\le d_Y(f(x),f(y))\le (1+\e){\cdot}d_X(x,y)$$for all $x,y\in X$.
\end{definition}


	

\begin{definition} A metric space is called
{\em Lipschitz countable dense homogeneous} (briefly, {\em LCDH}) if for all $\varepsilon>0$ and countable dense sets $A,B\subseteq X$, there exists an $\e$-isometry $f:X\to X$ such $f[A]=B$.
\end{definition}
	
By a result of Dijkstra \cite{Dijkstra}, all Banach spaces are LCDH.
	
	

A function $f:X\to Y$ between metric spaces $(X,d_X)$ and $(Y,d_Y)$ is called {\em contractive} if there exists a real number $c<1$ such that 
$d_Y(f(x),f(y))\le c\cdot d_X(x,y)$ for all $x,y\in X$.

\begin{lemma}\label{l:2} A countable metric space admits a contractive bijection and hence is not plastic whenever its completion is LCDH and admits a contractive bijection.
\end{lemma}
	
\begin{proof} Assume that the completion $\overline X$ of a countable metric space $X$ is LCDH and admits a contractive bijection $f:\overline X\to \overline X$. Let $\mathsf d$ be the metric of the complete metric space $\overline X$. Find a positive real number $\e$ such that $\mathsf d(f(x),f(y))\le (1-\e)\cdot\mathsf d(x,y)$ for all $x,y\in X$.

Since $f$ is continuous and surjective, the set $f[X]$ is dense in $\overline X$. Since $X$ is LCDH, there exists an $\e$-isometry $g$ of $\overline X$ such that $g[f[X]]=X$. Then $h:=g\circ f{\restriction}_X$ is a bijection of $X$ such that a
$${\sf d}(h(x),h(x))={\sf d}(g(f(x)),g(f(y))\le (1+\e)\cdot d(f(x),f(y))\le (1+\e)\cdot(1-\e)\cdot {\sf d}(x,y)=(1-\e^2)\cdot {\sf d}(x,y)$$for all $x,y\in X$,
witnessing that the map $h$ is a contractive bijection of $X$.
\end{proof}

Now we are able to present the {\em proof of Theorem~\ref{t:main1}.} Let $X$ be a countable dense set in a normed space $Y$. Let $\overline Y$ be the completion of the normed space $Y$. Observe that the function $\overline Y\to\overline Y$, $y\mapsto\frac12y$, is a contractive bijection of the Banach space $\overline Y$. By \cite{Dijkstra}, the Banach space $\overline Y$ is LCHD. By Lemma~\ref{l:2}, the countable dense subspace $X$ of $\overline Y$ admits a contractive bijection and hence is not plastic.\hfill$\square$

\section{Construction of a plastic space in Example~\ref{e:Bielas}}

In \cite{Bielas}, Wojciech Bielas constructed an example of a plastic dense subspace $X\subseteq \IR$ of cardinality $|X|=\aleph_1$, thus answering an open problem posed in an earlier version of this paper. In this section we present a simplified version of the construction of Bielas. 

Fix any bijective map $z:\IN\to\IZ$, $z:n\mapsto z_n$, and any strictly decreasing sequence $(\e_n)_{n\in\IN}$ such that $\e_1\le\frac12$ and $\lim_{n\to\infty}\e_n=0$. Choose a set $A\subseteq\IR$ such that for every nonempty open set $U\subseteq\IR$, the intersection $A\cap U$ has cardinality $\aleph_1$. Consider the dense subspace 
$$X:=\bigcup_{n\in\IN}(A\setminus[z_n-\e_n,z_n+\e_n])\cup ([z_n-\e_n,z_n+\e_n]\cap\IQ)$$ of the real line.

\begin{proposition} Every nonexpansive bijection $f:X\to X$ is the identity map of $X$, which implies that the metric space $X$ is plastic.
\end{proposition}

\begin{proof} A point $x\in X$ is called a {\em condensation point} of $X$ if each neighborhood $O_x$ of $x$ in $X$ has cardinality $|O_x|=|X|=\aleph_1$.  Consider the countable subset $Q\defeq\bigcup_{n\in\IN}([z_n-\e_n,z_n+\e_n]\cap\IQ)$ of $X$ and observe that a point $x\in X$ is a condensation point of the space $X$ if and only if $x\notin Q$. By the bijectivity and continuity of the function $f$, for any condensation point $x$ of $X$, its image $f(x)$ is a condensation point of $X$, which implies that  $f^{-1}[Q]\subseteq Q$.

\begin{claim}\label{cl:f(z)=z} $f(z_n)=z_n$ for all $n\in\IN$.
\end{claim}

\begin{proof} Assume that for some $n\in\IN$ we know that $f(z_k)=z_k$ for all $k<n$. Consider the point $x\defeq f^{-1}(z_n)\in f^{-1}[Q]\subseteq Q$ and observe that $x\in (z_m-\e_m,z_m+\e_m)\cap \IQ$ for some $m\in\IN$. 

Assuming that $m<n$, we can apply the inductive hypothesis and conclude that $f(z_m)=z_m$ and hence $|z_n-z_m|=|f(x)-f(z_m)|\le|x-z_m|<\e_m\le \e_1\le\frac12$ and hence $z_n=z_m$ and $n=m$, which is a contradiction showing that $m\ge n$. 

We claim that $m=n$. To derive a contradiction, assume that $m>n$. If $x\in (z_m-\e_m,z_m]$, then the nonexpansivity of $f$ ensures that $|f(z_m-\e_m)-z_n|=|f(z_m-\e_m)-f(x)|\le|x-(z_m-\e_m)|\le\e_m<\e_n$ and hence $f(z_m-\e_m)\in (z_n-\e_n,z_n+\e_n)$ which is not possible because $z_m-\e_m$ is a condensation point of the set $X$ whereas any point of the interval $(z_n-\e_n,z_n+\e_n)\cap\IQ$ is not condensation in $X$. This contradiction shows that $x\notin [z_m-\e_m,z_m]$. By analogy we can show that $x\notin[z_m,z_m+\e_m]$ and hence $m=n$.  Therefore, $x\in [z_n-\e_n,z_n+\e_n]$. Assuming that $x\in[z_n-\e_n,z_n)$, we conclude that   $|f(z_n-\e_n)-z_n|\le x-(z_n-\e_n)<\e_n$ and hence $f(z_n-\e_n)\in (z_n-\e_n,z_n+\e_n)$ which is not possible because $z_n-\e_n$ is a condensation point of the set $X$ whereas any point of the interval $(z_n-\e_n,z_n+\e_n)\cap\IQ$ is not condensation in $X$. This contradiction shows that $x\notin[z_n-\e_n,z_n)$. By analogy we can prove that $x\notin(z_n,z_n+\e_n]$. Therefore, $x=z_n$ and hence $f(z_n)=z_n$. The Principle of Mathematical induction ensures that $f(z_n)=z_n$ for all $n\in\IN$.
\end{proof}

Claim~\ref{cl:f(z)=z} implies that $f{\restriction}_\IZ$ is the identity map of $\IZ$. To see that the map $f$ is the identity map of $X$, fix any point $x\in X$ and consider the integer numbers 
$$\lfloor x\rfloor\defeq\max\{n\in\IZ:n\le x\}\quad\mbox{and}\quad\lceil x\rceil\defeq \min\{n\in\IZ:x\le n\}.$$ Observe that $f(z)=z$ for $z\in\{\lfloor x\rfloor,\lceil x\rceil\}$. The nonexpansivity of the function $f$ ensures that $f(x)$ belongs to the singleton $\{y\in\IR:|y-\lfloor x\rfloor|\le|x-\lfloor x\rfloor|\}\cap\{y\in\IR:|y-\lceil x\rceil|\le|x-\lceil x\rceil|\}=\{x\}$ and hence $f(x)=x$.
\end{proof}

	

\section{Cantor sets in analytic spaces}\label{s:Cantor}

A metrizable topological spaces is called {\em analytic} if it is a continuous image of a Polish space. A topological space is {\em Polish} if it is separable and its topology is generated by a complete metric.

In the proofs of Theorems~\ref{t:main2} and \ref{t:main3} we shall extensively exploit the following classical fact, due to Souslin \cite[29.1]{Kechris}.

\begin{proposition}\label{p:Cantor-set} Every uncountable analytic space contains a Cantor set.
\end{proposition}

By a {\em Cantor set} we understand any topological copy of the Cantor cube $\{0,1\}^\w$. By the classical Brouwer's Theorem~\cite[7.4]{Kechris}, a nonempty topological space is a Cantor set if and only if it is zero-dimensional, compact, metrizable, and has no isolated points. We recall that a topological space is {\em zero-dimensional} if it has a base of the topology, consisting of clsed-and-open sets. 

\section{Proof of Theorem~\ref{t:main2}}\label{sec2}

Given any $k$-crowded separable metric space $Y$, we will construct a dense subspace $X$ of $Y$ such that every non-expansive bijection of $X$ is the identity. This property of $X$ ensures that it is plastic. We recall that a space is {\em $k$-crowded} if every nonempty open sets in this space contains an uncounable compact subset.
\smallskip

	
	


		\textit{Construction of $X$}. Since the metric space $Y$ is separable, its topology is second countable and hence admits a countable base $\mathfrak{B} = \{B_n: n \in \omega\}$ consisting of non-empty open sets. Since $Y$ is $k$-crowded,  every set $B_n$ contains an uncountable compact set $K_n$. Consider the $\sigma$-compact dense subspace $K = \bigcup_{n \in \omega} K_n$ in $X$.  The plastic dense space $X$ will be constructed as a dense subspace of the $\sigma$-compact space $K$.

Let $\mathcal C$  be the family of all Cantor sets in $K$. Proposition~\ref{p:Cantor-set} implies that the uncountable $\sigma$-compact space $K$ contains a Cantor set, and hence the family $\mathcal C$ has cardinality of continuum $\mathfrak c$. Write the cardinal $\mathfrak c=[0,\mathfrak c)$ as the union $\mathfrak c=\Omega_0\cup\Omega_1$ of two disjoint sets of cardinality $\mathfrak c$ such that $0\in\Omega_0$. The family $\mathcal C$ has cardinality $\mathfrak c$ and hence can be written as $\{C_\alpha:\alpha\in\Omega_0\}$.

Let $\overline{Y}$ be the completion of the metric space $K$, and let $\mathcal{F}$ 
be the set of all non-expansive self-maps of the space $\overline{Y}$. Since any map $f \in \mathcal{F}$ is continuous, it is uniquely determined by its restriction to a dense subset of $\overline{Y}$. The separability of $Y$ implies that $|\mathcal{F}| \leq \mathfrak{c}$ and hence the family $\mathcal F$ can be written as $\{f_{\alpha} : \alpha \in \Omega_1\}$. 	
		
By transfinite induction, for every $\alpha < \mathfrak{c}$ we construct subsets $X_\alpha$, $V_\alpha$ of $K$ so that the following conditions are satisfied:
		\begin{enumerate}
			\item $X_{<\alpha}\defeq\bigcup_{\beta<\alpha}X_\beta\subseteq X_\alpha$, $V_{<\alpha}\defeq\bigcup_{\beta<\alpha}V_\beta \subseteq Z_\alpha$ and $X_\alpha \cap V_\alpha = \emptyset$;
			\item $|X_\alpha| \leq \w+\alpha$ and $|V_\alpha| \leq \w+\alpha;$
			\item if $\alpha \in \Omega_0 \setminus \{0\}$ then $X_\alpha \cap C_\alpha \not = \emptyset;$
			\item if $\alpha \in \Omega_1$ and the set $A_\alpha \defeq \{x \in K : f_\alpha(x) \not \in (X_{<\alpha} \cup \{x\})\}$ has cardinality $\mathfrak{c}$, then there exists a point $x \in X_\alpha$ such that $f_\alpha(x) \in V_\alpha.$
		\end{enumerate}
		
Put $X_0 = V_0 = \emptyset$ and suppose that for some ordinal $\alpha<\mathfrak c$, we have completed the inductive construction for all ordinals $\beta < \alpha$. The inductive conditions (1) and (2) ensure that the sets $X_{<\alpha}\defeq\bigcup_{\beta<\alpha}X_\beta$ and $V_{<\alpha}\defeq\bigcup_{\beta<\alpha}V_\beta$ have cardinality 
$|X_{<\alpha}| \leq \w+\alpha<\mathfrak{c}$ and $|V_{<\alpha}|\leq\w+\alpha< \mathfrak{c}$. Consider the following possible cases.
\smallskip

		Case 1: $\alpha \in \Omega_0$. Since $|C_\alpha| = \mathfrak{c}$, we can pick a point $x \in C_\alpha \setminus V_{<\alpha}$, and put $X_\alpha \defeq  V_{<\alpha} \cup \{x\}$ and $V_\alpha\defeq V_{<\alpha}$. We easily check that the sets $X_\alpha, V_\alpha$ are as required.
\smallskip
		
		Case 2: $\alpha \in \Omega_1$ and $\left| A_\alpha \right| < \mathfrak{c}$. Put $X_\alpha\defeq X_{<\alpha}$ and $V_\alpha \defeq V_{<\alpha}$. We easily check that the sets $X_\alpha, V_\alpha$ are as required.
\smallskip
		
		Case 3: $\alpha \in \Omega_1$ and $\left| A_\alpha \right| = \mathfrak{c}$. Since  $|X_{<\alpha}| < \mathfrak{c}$ and $|V_{<\alpha}| < \mathfrak{c}$, there exists  a point $x \in A_\alpha \setminus V_{<\alpha}$. Put $X_\alpha\defeq X_{<\alpha} \cup \{x\}$ and $V_\alpha\defeq V_{<\alpha} \cup \{f_\alpha(x)\}$. Since $X_{<\alpha} \cap V_{<\alpha} = \emptyset$ and $x \in A_\alpha \setminus V_{<\alpha}$, we get $X_\alpha \cap V_\alpha = \emptyset.$ By easy checks of the other conditions, we get that the sets $X_\alpha, V_\alpha$ are as required.
\smallskip
	
After completing the inductive construction, consider the set $X\defeq \bigcup_{\alpha < \mathfrak{c}}X_\alpha$. We claim, that $X$  is dense plastic subspace of $Y$, and moreover $X$ admits no non-identity non-expansive self-maps.
\smallskip		
		
\noindent\textit{The density of $X$ in $Y$.} Let $U$ be any non-empty open set in $Y$. Then for some $n \in \omega$ we have $K_n \subseteq B_n \cap K \subseteq U$, where $K_n$ is an uncountable compact set in $K$. By Proposition~\ref{p:Cantor-set}, $K_n$ contains a Cantor set. By property $(3)$, the set $X$ intersects every Cantor set in $K$, which yields that $Y \cap U \not = \emptyset$, hence $X$ is dense in $Y$. Then the completion $\overline Y$ of the metric space $Y$ is also the completion of the metric space $X$.
\smallskip
		
\noindent\textit{The plasticity of $X$.} Given any non-expansive bijection $f$ of $X$, we shall prove that $f$ is the identity isometry of $X$.  Since $f$ is non-expansive, it is uniformly continuous. By Theorem 4.3.17 in \cite{Engelking}, $f$ has a unique continuous extension $\bar{f}: \overline{Y} \to \overline{Y}$ to the completion $\overline Y$ of the metric space $X$.  Since $\bar f$ is non-expansive, it belongs to the family $\mathcal F$ and hence $\bar f=f_\alpha$ for some $\alpha\in \Omega_1$.	
	
Assuming that the set $A_\alpha\defeq \{x \in K : f_\alpha(x) \not \in (X_{<\alpha} \cup \{x\})\}$ has cardinality continuum, we can apply the inductive condition (4) and find a point $x \in X_\alpha \subseteq X$ such that $f_\alpha(x) \in V_\alpha$. Then $f_\alpha(x)\in X\cap Z$, which contradicts the inductive condition (1). This contradiction shows that $|A_\alpha|<\mathfrak c$.

Assume that for some $y \in X_{<\alpha}$, the set $f_\alpha^{-1}(y) \cap K$ is uncountable. Since $f_\alpha^{-1}(y) \cap K$ is a closed subspace of the $\sigma$-compact space $K$, it is $\sigma$-compact. By Proposition~\ref{p:Cantor-set}, it contains a Cantor set and hence continuum pairwise disjoint Cantor sets, each intersecting the set $X$, by the inductive condition (3). Then $| f_\alpha^{-1}(y) \cap X| = \mathfrak{c}$, which contradicts the bijectivity of the map $f:X \to X$. This contradiction shows that for every $y\in X_{<\alpha}$, the set $f_\alpha^{-1}(y) \cap K$ is at most countable.  From property $(1)$ we then get $| f_\alpha^{-1}[X_{<\alpha}]\cap K| \leq |X_{<\alpha}| \cdot \w\leq |\w+\alpha|  \cdot \w < \mathfrak{c}$. Then the set $B\defeq A_\alpha\cup (f^{-1}_\alpha[X_{<\alpha}]\cap K)$ has cardinality $|B|<\mathfrak c$ and hence $K\setminus B$ is dense in the $k$-crowded space $K$ (because every uncountable compact metrizable space contains a Cantor set and hence has cardinality $\mathfrak c$).
Observe that $f_\alpha(x)=x$ for all $x\in K\setminus B$. The density of the set $K\setminus B$  in $\overline Y$ ensures that the continuous map $f_\alpha$ is the identity map of $\overline Y$. Then $f=f_\alpha{\restriction}_X$ is the identity map of $X$, witnessing that the metric space $X$ is plastic.\hfill$\square$

\section{Metric and convex intervals in metric and normed spaces}\label{sec2a}

In this section prove some auxiliary results on metric and convex intervals in metric and normed spaces. These auxiliary resuls will be used in the proof of Theorem~\ref{t:main3}, presented in the next section.	
	
\begin{definition}
A metric space $(X, {\sf d})$ is called a \emph{metric interval} with endpoints  $a, b \in X$ (denoted by $([a, b], {\sf d})$) if $ {\sf d}(a, b) = \min\{{\sf d}(a, x) + {\sf d}(x, y) + {\sf d}(y, b), {\sf d}(a, y) + {\sf d}(y, x) + {\sf d}(x, b)\}$ for all $x,y\in X$.
\end{definition}

\begin{example}
Every  nonempty compact subset $K$ of the real line is a metric interval with endpoints $a= \min K$ and $b=\max K$.
\end{example}

\begin{lemma}\label{l:3}
Let $([a, b], {\sf d})$ be a metric interval and $(Y, {\sf d_Y})$ be a metric space. Every non-expansive map $f: [a,b] \to Y$ with ${\sf d}(a, b) = {\sf d_Y}(f(a), f(b))$ is an isometry.
\end{lemma}

\begin{proof}
Take arbitrary points $x, y \in [a,b]$ and assume that ${\sf d}(x, y) > {\sf d_Y}(f(x), f(y))$. Using the triangle inequality in the metric space $Y$ and the non-expansive property of the map $f$, we get the following inequality
$${\sf d_Y}(f(a), f(b)) \leq {\sf d_Y}(f(a), f(x)) + {\sf d_Y}(f(x), f(y)) + {\sf d_Y}(f(y), f(b)) < {\sf d}(a,x) + {\sf d}(x, y) + {\sf d}(y, b).$$
By the similar argument we can also obtain the inequality
$${\sf d_Y}(f(a), f(b)) \leq {\sf d_Y}(f(a), f(y)) + {\sf d_Y}(f(y), f(x)) + {\sf d_Y}(f(x), f(b)) < {\sf d}(a,y) + {\sf d}(y, x) + {\sf d}(x, b).$$
Together these results show
$${\sf d}(a, b) = {\sf d_Y}(f(a), f(b)) < \min\{{\sf d}(a, x) + {\sf d}(x, y) + {\sf d}(y, b), {\sf d}(a, y) + {\sf d}(y, x) + {\sf d}(x, b)\} = {\sf d}(a, b),$$
which is a contradiction that implies ${\sf d}(x, y) \leq {\sf d_Y}(f(x), f(y))$. Then, the map $f$ is non-contractive and non-expansive at the same time, so it is an isometry.
\end{proof}
	\begin{corollary}
		Let $([a,b],{\sf d})$ be a metric interval. The map ${\sf d}(a, \cdot): [a,b] \to \mathbb{R}$,  ${\sf d}(a, \cdot)\colon x \rightarrow {\sf d}(x, a) $ is an isometry.
	\end{corollary}
	\begin{definition}
		A metric space $(X,{\sf d})$ is \emph{strictly convex} if for all points $x,y\in X$ and all positive real numbers $\alpha,\beta$ with $\alpha+\beta={\sf d}(x,y)$ there exists a unique point $z\in X$ such that ${\sf d}(x,z)=\alpha$ and ${\sf d}(z,y)=\beta$.
	\end{definition}
	
	\begin{lemma}\label{rem:1}
		Every convex subspace $C$ of a strictly convex normed space $X$ is strictly convex.
	\end{lemma}
	\begin{proof}
		Take arbitrary points $x,y\in C$ and positive real numbers $\alpha,\beta$ such that $\alpha+\beta=\|x-y\|$. By the strict convexity of $X$, there exists a unique point $z\in X$ such that $\|x-z\|=\alpha$ and $\|z-y\|=\beta$. On the other hand, the convex combination $c\defeq \frac{\beta}{\alpha+\beta}{\cdot}x+\frac{\alpha}{\alpha+\beta}{\cdot}y\in C$ also has $\|x-c \|=\alpha$ and  $\|c-y\| = \beta$. The uniqueness of $z$ implies $z=c\in C$, witnessing that the convex subspace $C$ of $X$ is strictly convex.
	\end{proof}
	
\begin{lemma}\label{l:4}
Let $([a,b], {\sf d})$ be a metric interval and $(Y, {\sf d_Y})$ be a strictly convex metric space. Then two isometric maps $f_1,f_2: [a,b] \to Y$ that satisfy $f_1(a) = f_2(a), f_1(b) = f_2(b)$ are identical on $[a, b]$ in the sense $f_1(x) = f_2(x) $ for all $x \in [a,b]$.
\end{lemma}

\begin{proof}
Take an arbitrary point $x \in [a, b]$ and put $\alpha\defeq {\sf d}(a,x)$ and $\beta\defeq {\sf d}(x,b)$. Observe that $${\sf d}(a,b) = \min\{{\sf d}(a,z) + {\sf d}(z,z) + {\sf d}(z,b), {\sf d}(a,z) + {\sf d}(z,z) + {\sf d}(z,b)\} = {\sf d}(a,z) + {\sf d}(z,b)=\alpha+\beta.$$
Consider the points $A\defeq f_1(a)=f_2(a)$ and $B\defeq f_1(b)=f_2(b)$ in $Y$. For each $i\in\{1,2\}$, we have ${\sf d}_Y(A,f_i(x))={\sf d}(a,x)=\alpha$ and ${\sf d}_Y(f_i(x),B)={\sf d}(x,b)=\beta$. Since ${\sf d}_Y(A,B)={\sf d}(a,b)=\alpha+\beta$, the strict convexity of $Y$ ensures that $f_1(x)=f_2(x)$.
\end{proof}
	
\begin{lemma}\label{l:5}
For any points $a,b\in X$ of a normed space $(X, \|\cdot\|)$, the convex interval $$[a, b] \defeq \{(1-t){ \cdot} a + t{\cdot} b: t \in [0, 1]\}\subset X$$ endowed with 
 the metric ${\sf d}(x, y) = \| x - y \|$, is a metric interval with endpoints $a, b$.
\end{lemma}

\begin{proof}
Take any $x, y$ from the convex interval $[a, b]$. Then for some $\alpha, \beta \in [0, 1]$ we have $x = (1-\alpha) \cdot a + \alpha \cdot b$ and $y = (1-\beta) \cdot a + \beta \cdot b$. Substituting it while calculating the following sum of distances will result in equality ${\sf d}(a,x) + {\sf d}(x,y) + {\sf d}(y, b) = \|b - a\| \cdot (1 + \alpha - \beta + |\beta - \alpha|)$. In the same way, we obtain ${\sf d}(a,y) + {\sf d}(y,x) + {\sf d}(x, b) = \|b - a\| \cdot (1 + \beta - \alpha + |\beta - \alpha|)$.
		Both when $\alpha \leq \beta$ or $\beta < \alpha$, these equalities result in
		$$\min\{{\sf d}(a,x) + {\sf d}(x,y) + {\sf d}(y, b), {\sf d}(a,y) + {\sf d}(y,x) + {\sf d}(x, b)\} = \|b-a\| = {\sf d}(a, b).$$
	\end{proof}
	


\section{Proof of Theorem~\ref{t:main3}}\label{sec3}

Given a strictly convex separable metric abelian group $(X,+,0,{\sf d})$, we shall construct a plastically rigid dense subgroup $H$ in $X$. For a subset $A\subseteq X$ we denote by $\langle A\rangle$ the group hull of $A$, i.e., the smallest subgroup of $X$ that containes the set $A$. For two subsets $A,B$ of the group $X$, put $A+B\defeq\{a+b:a\in A,\;b\in B\}$ and $A-B\defeq\{a-b:a\in A,\;b\in B\}$.

By \cite{BM}, the strictly convex metric abelian group $X$ admits a unique binary operation $\cdot:\IR\times X\to X$, $\cdot:(t,x)\mapsto tx$, that turns $X$ into a normed space, endowed with the norm $\|x\|={\sf d}(0,x)$.
So, we shall think of $X$ as a strictly convex normed space. If $X=\{0\}$, then the trivial subgroup $H=\{0\}$ is plastically rigid. So, assume that the normed space $X$ contains more than one point.
\smallskip

\noindent\textit{Inductive construction of the subgroup $H$}. Fix a countable set set $Q = \{x_i\}_{i \in \omega}$ in the separable normed space $X$, and consider its linear hull $$K := \operatorname{span}(Q) = \bigcup_{n \in \omega}\bigcup_{m\in\w} \big\{\textstyle{\sum_{i\in n}\alpha_i x_i}:(\alpha_i)_{i\in n}\in[-m,m]^n\big\}$$in $X$. By Lemma~\ref{rem:1}, $K$ is a strictly convex dense linear subspace of the strictly convex normed space $X$. 
Since all sets $\big\{\sum_{i\in n}\alpha_i x_i:(\alpha_i)_{i\in n}\in[-m,m]^n\big\}$, $n,m\in\w$, are compact, the space $K$ is $\sigma$-compact. The plastically rigid dense subgroup $H$ of the group $X$ will be constructed as a $k$-dense subgroup of the $\sigma$-compact linear space $K$. The $k$-density of $H$ means that it intersects every uncountable compact subset of $K$.

Let $\overline X$ be the completion of the normed space $X$. By the density of $K$ in $X$, the Banach space $\overline X$ is also a completion of the $\sigma$-compact normed space $K$. 
By the separability of $\overline X$, the set $\mathcal{F}$ of all non-expansive self-maps of the $\overline X$ has cardinality of continuum. By the same reason, the family $\mathcal C$ of all Cantor sets in $K$ also has cardinality of continuum.

Write the cardinal $\mathfrak c=[0,\mathfrak c)$ as the union of two disjoint sets $\Omega_0$ and $\Omega_1$ of cardinality $|\Omega_0|=|\Omega_1|=\mathfrak c$ such that $0\in\Omega_0$. Since $\mathcal F$ and $\mathcal C$ are sets of cardinality continuum, there exist their enumerations $\{C_\alpha:\alpha\in\Omega_0\}=\mathcal C$ and $\{f_\alpha:\alpha\in\Omega_1\}=\mathcal F$ such that $0\in C_0$.  

For every $\alpha\in \Omega_1$, the preimage $f^{-1}[K]$ of the $\sigma$-compact set $K$ in $\overline X$ can be written as the union of countably many closed sets in $\overline X$. Then its intersection $K_\alpha\defeq f^{-1}[K]\cap K$ with the $\sigma$-compact set $K$ remains $\sigma$-compact.

By transfinite induction, for every $\alpha < \mathfrak{c}$ we shall construct a subgroup $H_\alpha$ and a subset $V_\alpha$ of $K$ so that the following conditions are satisfied:
		\begin{enumerate}
			\item $H_{<\alpha}\defeq\bigcup_{\beta<\alpha}H_\beta \subseteq H_\alpha$, $V_{<\alpha}\defeq\bigcup_{\beta<\alpha}V_\beta \subseteq V_\alpha$, and $H_\alpha \cap V_\alpha = \emptyset$;
			\item $\left| H_\alpha \right| \leq\w+\alpha, \left| V_\alpha \right| \leq \w+\alpha;$
			\item if $\alpha \in \Omega_0$ then $H_\alpha \cap C_\alpha \not = \emptyset;$
			\item if $\alpha \in \Omega_1$ and the set $A_\alpha = \{x \in K_\alpha : f_\alpha(x) \not \in \langle H_{<\alpha} \cup \{x_\alpha\}\rangle \}$ has cardinality $\mathfrak{c}$, then there exists a point $x_\alpha \in H_\alpha$ such that $f_\alpha(x_\alpha) \in V_\alpha$.
				
		\end{enumerate}
		Put $H_0 = \{0\}, V_0 = \emptyset$ and suppose that for some ordinal $\alpha<\mathfrak c$, we completed the inductive construction for all ordinals $\beta < \alpha$. By the inductive continuous (1) and (2), the set $H_{<\alpha}\defeq\bigcup_{\beta<\alpha}H_\beta$ and $V_{<\alpha}\defeq\bigcup_{\beta<\alpha}V_\beta$ have cardinality $|H_{<\alpha}|\leq\w+\alpha<\mathfrak{c}$ and $|V_{<\alpha}|\leq \w+\alpha< \mathfrak{c}$. Consider the following possible cases.
\smallskip	

Case 1: $\alpha \in \Omega_0$. Since $K$ is a linear space, for each $n\in\IZ\setminus\{0\}$, the set $$\tfrac{1}{n}(H_{<\alpha} + V_{<\alpha}) = \{x \in K: (\exists y \in H_{<\alpha})(\exists v \in V_{<\alpha}): n \cdot x = y + v\}$$ has cardinality at most $|H_{<\alpha}|\cdot|V_{<\alpha}|\le|\w+\alpha|<\mathfrak{c}$. Since $\left| C_\alpha \right| = \mathfrak{c}$, we can pick a point $x_\alpha \in C_\alpha \setminus (\bigcup_{n \in \mathbb{Z} \setminus \{0\}}\frac{1}{n}(V_{<\alpha} + H_{<\alpha})).$ It is clear that the subgroup $H_\alpha \defeq\langle H_{<\alpha} \cup \{x_\alpha\} \rangle$ and the subset $V_\alpha\defeq V_{<\alpha}$ of $K$ satisfy the inductive conditions (2)--(4). It remains to show that  $H_\alpha\cap V_\alpha=\varnothing$. In the opposite case, we can find a point $v \in H_\alpha \cap V_\alpha\subseteq H_\alpha\defeq\langle H_{<\alpha} \cup \{x\} \rangle = \bigcup_{n \in \mathbb{Z}}\left (H_{<\alpha} + n \cdot x \right)$, which can be written as $v=h+n{\cdot}x_\alpha$ for some  $h\in H_{<\alpha}$ and $n\in\IZ$. Since the sets $H_{<\alpha}\ni h$ and $V_{<\alpha}=V_\alpha\ni v$ are disjoint, the number $n$ is not zero. Then $x_\alpha\in \frac1n(v-h)$, which contradicts the choice of the point $x_\alpha$. This contradiction shows that the sets $H_\alpha$ and $V_\alpha$ are disjoint, and hence satisfy the inductive conditions (1)--(4).
\smallskip
		
Case 2: $\alpha \in \Omega_1$ and $\left| A_\alpha \right| < \mathfrak{c}$. In this case, put  $H_\alpha\defeq H_{<\alpha}$ and $V_\alpha\defeq V_{<\alpha}$, and observe that so defined sets $H_\alpha$ and $V_\alpha$ satisfy the inductive conditions (1)--(4).
\smallskip
		
Case 3: $\alpha \in \Omega_1$ and $| A_\alpha| = \mathfrak{c}$. In this case we can pick a point $x_\alpha \in A_\alpha$ that does not belong to the set $H_{<\alpha} \cup \bigcup_{n \in \mathbb{Z} \setminus \{0\}}\frac{1}{n}(H_{<\alpha} + V_{<\alpha})$ (that has cadinality $<\mathfrak c=|A_\alpha|$). Consider the subgroup $H_\alpha\defeq \langle H_{<\alpha} \cup \{x_\alpha\} \rangle$ and the subset $V_\alpha \defeq V_{<\alpha} \cup \{f_\alpha(x_\alpha)\}$ of $K$. Observe, that $x_\alpha \in K_\alpha$, which implies that $f_\alpha(x_\alpha) \in K$ and hence $V_\alpha$ and $H_\alpha$ are subsets of the linear space $K$. By the similar argument to the one in the Case 1, we check that the group $H_\alpha$ and the set $V_\alpha$ are as required.
\smallskip

After completing the inductive construction, consider the subgroup $H \defeq \bigcup_{\alpha < \mathfrak{c}}H_\alpha$ of $K$. We claim that the set $H$ is dense in the groups $K$ and $X$. For any nonempty open set $U\subseteq X$, by the density of $K$ in $X$, there exists a point $x \in U\cap K\setminus\{0\}$. The set $L = \{t\cdot x : t\in \mathbb{R}\}$ is a topological copy of the real line and hence contains a Cantor set $C\subseteq L\cap U\subseteq K$. Since $H$ intersects all Cantor sets in $K$ by construction, we get $\emptyset\ne H \cap C\subseteq H\cap U$, witnessing that the subgroup $H$ is dense in $K$ and $X$. It remains that the metric group $H$ is plastically rigid (and hence is a plastic metric space).
\smallskip
		
\noindent
\textit{The plastic rigidity of $H$.} Given any non-expansive bijection $f$ of $H$, we need to show that it is of the form $f(x)=a+s\cdot x$ for some $a\in H$ and $s\in\{-1,1\}$. By the density of $H$ in $X$, the completion $\overline X$ of $X$ is also a completion of $H$. By the uniform continuity, the non-expansive map $f$ admits a unique continuous extension $\bar f:\overline X\to\overline X$, see  Theorem 4.3.17 in \cite{Engelking}. It is easy to see that $\bar f$ is non-expansive and hence $\bar f=f_\alpha$ for some $\alpha\in\Omega_1$. Note that $f_\alpha[H]=f[H]=H\subseteq K$ and hence $H \subseteq f_\alpha^{-1}(H) \subseteq f_\alpha^{-1}(K)$, which implies that $H \subseteq K_\alpha$. Then the restriction $h_\alpha\defeq f_\alpha{\restriction_{K_\alpha}}:K_\alpha\to K$ is also an extension of the map $f$.
		
We claim that the set $A_\alpha\defeq \{x \in K_\alpha : f_\alpha(x) \not \in \langle H_{<\alpha} \cup \{x\}\rangle \}$ has cardinality $| A_\alpha | < \mathfrak{c}$. In the opposite case, by property $(4)$ there exists a point $x_\alpha \in H_{\alpha} \subseteq H$ such that $f(x_\alpha)=f_{\alpha}(x_\alpha) \in V_{\alpha}\subseteq K\setminus H$, which contradicts the choice of the bijection $f:H\to H$. This contradiction shows that $| A_\alpha | < \mathfrak{c}$.
		
Next, we show that for every $y \in H_{<\alpha}$, the preimage $h_{\alpha}^{-1}(y)$ is at most countable. Since $h_{\alpha}^{-1}(y)$ is a closed subspace of the $\sigma$-compact space $K_\alpha\defeq K\cap f_\alpha^{-1}[K]$, it is $\sigma$-compact. Hence by Proposition~\ref{p:Cantor-set}, it contains two disjoint Cantor sets. Then from property $(3)$ we conclude that $\left | h_{\alpha}^{-1}(y) \cap H \right |\ge 2$. Since $h_{\alpha}$ extends a bijective map $f: H \to H$, it is a contradiction showing that $h_{\alpha}^{-1}(y)$ is at most countable for all $y \in H_{<\alpha}.$ 

From property $(2)$ we then get $\left | h_{\alpha}^{-1}[H_{<\alpha}] \right | \leq \left | H_{<\alpha} \right | \cdot \w \leq \left | \alpha + \omega \right | \cdot \w < \mathfrak{c}$. 	Observe also that for all $n \in \mathbb{Z} \setminus \{0, 1, -1\}$ and $h \in H_{<\alpha}$ the set $\{x \in K_\alpha : h_{\alpha}(x) = h+n \cdot x\}$ contains at most one element. Indeed, otherwise, for two distinct points $x,y$ from this set the following contradiction to the fact that $h_{\alpha}$ is non-expansive would arise
		$${\sf d}(h_{\alpha}(x), h_{\alpha}(y)) = {\sf d}(h+n \cdot x, h+n \cdot y) = {\sf d}(n \cdot x, n \cdot y) = |n| \cdot {\sf d}(x,y) > {\sf d}(x,y).$$ Hence, for every $n \in \mathbb{Z} \setminus \{0, 1, -1\}$ the cardinality of $\{x \in K_\alpha : h_{\alpha}(x) \in n \cdot x + H_{<\alpha}\}$ does not exceed $|H_{<\alpha}|<\mathfrak{c}$.
		
Consider the set $B\defeq\{x\in K_\alpha:h_\alpha(x)\notin \{-x,x\}+H_{<\alpha}\}$ and observe that $$B\subseteq A_\alpha\cup h_\alpha^{-1}[H_{<\alpha}]\cup\bigcup_{n\in\IZ\setminus\{-1,0,1\}}\{x\in K_\alpha:h_\alpha(x)\in H_{<\alpha}+n\cdot x\}$$and hence $|B|\le |A_\alpha|+|h_{\alpha}^{-1}[H_{<\alpha}]|+|H_{<\alpha}|\cdot\w<\mathfrak c$.
\smallskip
		
The plastic rigidity of the group $H$ will follow as soon as we find a point $a\in H_{<\alpha}$  and sign $s\in\{-1,+1\}$ such that  the set $$F_a^s\defeq \{x \in K : f_\alpha(x) = a+s{\cdot}x\}$$coincides with $K$. To find such $a$ and $s$, we first establish some basic properties of the sets $F_a^s$ for arbitrary $a\in H_{<\alpha}$ and $s\in\{-1,+1\}$. 

\begin{claim}\label{cl:1}  For all $a\in H_{<\alpha}$ and $s\in\{-1,+1\}$, the set $F_a^s$ is a closed convex subset of the normed space $K$, and $F_a^s\subseteq K_\alpha$.
\end{claim}

\begin{proof} Taking into account that $a \in H_{<\alpha} \subseteq H \subseteq K$ and $K$ is a group, we conclude that $f_\alpha[F_a^s]=a+s{\cdot}F_a^s \subseteq K$ and hence $F_a^s\subseteq K_\alpha$. The continuity of the function $f_\alpha$ implies that the set $F_a^s$ is closed in $K$. It remains to prove that the set $F^s_a$ is convex in the linear space $K$.

Given any distinct points $x,y\in F_a^s$, we need to check that the convex interval $[x,y]\defeq\{tx+(1-t)y:t\in[0,1]\}$ in the linear space $K$ is contained in the set $F_a^s$. By Lemma $\ref{l:5}$ the convex interval $[x,y]$ is a metric interval with endpoints $x,y \in F_a^{s} \subseteq K_\alpha$. Then the set $[x,y]_\alpha\defeq [x,y] \cap K_\alpha$ is also a metric interval with endpoints $x,y $. Moreover, we have ${\sf d}(f_{\alpha}(x), f_{\alpha}(y)) = {\sf d}(s\cdot x + a, s\cdot y + a) =\|x-y\|={\sf d}(x,y).$ Hence, by Lemma $\ref{l:3}$ the map $f_{\alpha}{\restriction}_{[x,y]_\alpha}: [x,y]_\alpha \to K$ is then an isometry. On the other hand, the map $g: [x,y]_\alpha \to K$, $g:z\mapsto a+s{\cdot}z$, also is an isometry with $g(x) =a+s{\cdot}x=f_{\alpha}(x)$ and $g(y) =a+s{\cdot}y= f_{\alpha}(y)$. Since $K$ is strictly convex, Lemma \ref{l:4} ensures that $f_{\alpha}(z) = g(z) = a+s{\cdot}z$ for all $z \in [x,y]_\alpha$. This implies $[x,y]_\alpha \subseteq F_{a}^{s}$. Since the convex interval $[x,y]$ is a topological copy of the interval $[0,1] \subseteq \mathbb{R}$, every open non-empty subset of $[x,y] \subseteq K$ contains a Cantor set and hence intersects $H \subseteq K_\alpha$. Then, the set $[x,y]_\alpha=[x,y] \cap K_\alpha \subseteq F_a^s$ is dense in $[x,y]$. Since $F_a^{s}$ is closed in $K$, we get that the closure $\overline{[x,y]_\alpha} = [x,y]$ is contained in $F_a^{s}$, witnessing that the set $F_{a}^{s}$ is convex. 
\end{proof}

The definition of the sets $F_a^s$ implies the following claims.

\begin{claim}\label{cl:2} For any distinct points $a,b\in H_{<\alpha}$ and any sign $s\in\{-1,+1\}$, the sets $F_a^s$ and $F_b^s$ are disjoint.
\end{claim}

\begin{claim}\label{cl:2a} For any points $a,b\in H_{<\alpha}$, we have $F_a^{-1}\cap F_b^{+1}\subseteq\{\tfrac{a-b}2\}$.
\end{claim}

By a {\em line} in the normed space $K$, we understand the affine hull $\{(1-t){\cdot}x+t{\cdot}y:t\in\IR\}$ of any distinct points $x,y\in K$. It is easy to see that each line in $K$ is an isometric copy of the real line.

\begin{claim}\label{cl:3}  Every  line $L$ in the linear space $K$ is contained in the set $F_a^s$ for some point  $a\in H_{<\alpha}$ and sign $s\in\{-1,+1\}$.
\end{claim}

\begin{proof} By Claim~\ref{cl:1}, for all $a\in H_{<\alpha}$ and $s\in\{-1,+1\}$, the intersection $C_a^s\defeq L \cap F_a^s$ is a closed convex set in $L$.  Let $I_a^s$ be the interior of the set $C_a^s$ in $L$. Observe that the boundary $\partial C_a^s$ of the convex set $C_a^s$ in $L$ contains at most two points. Moreover, $|C_a^s|\le 1$ if and only if $I_a^s=\emptyset$. For every $s\in\{-1,+1\}$, consider the set $A^s\defeq \{a \in H_{<\alpha}:I_a^s\ne\emptyset\}$. By Claim~\ref{cl:2}, the indexed family $(I_a^s:a\in A^s)$ consists of pairwise disjoint nonempty open sets in the separable metric space $L$, which implies $|A^s|\le\w$.

Observe that the line $L$ is a topological copy of the real line $\IR$ and hence $L$ is a Polish space. Since the space $K_\alpha$ is $\sigma$-compact, the subspace $L\setminus K_\alpha$ is a Polish subspace of $L$ (being a $G_\delta$-set in $L$). By Proposition~\ref{p:Cantor-set}, each uncoutable Polish space contains a Cantor set. Since the group $H$ intersects all Cantor sets in $K$, the $G_\delta$-set $L\setminus K_\alpha$ is at most countable. Recall that the set $B=\{x\in K_\alpha:f_\alpha(x)\notin\{-x,x\}+H_{<\alpha}\}$ has cardinality $|B|<\mathfrak c$. Observe that $$L=(L\setminus K_\alpha)\cup (L\cap B)\cup\bigcup_{s\in \{-1,1\}}\bigcup_{a\in H_{<\alpha}}(I_a^s\cup\partial C^s_a)$$ and hence the closed subset
$F\defeq L\setminus\bigcup_{s\in\{-1,1\}}\bigcup_{a\in A^s}I_a^s$ of the line $L$ is contained in the set $(L\setminus K_\alpha)\cup(L\cap B)\cup\bigcup_{s\in\{-1,1\}}\bigcup_{a\in H_{<\alpha}}\!\partial C^s_a$, which has cardinality $<\mathfrak c$. By Proposition~\ref{p:Cantor-set}, the Polish space $F$ is at most countable.

We claim that the Polish space $F$ is empty. In the opposite case, the Baire Theorem ensures the existence of an isolated point $u$ in $F$. Then for some open ball $B(u,\e)$ in the line $L$ we have $B(u,\e)\cap F=\{u\}$. The complement $B(u,\e))\setminus\{u\}$ can be written as the disjoint union $U^-\cup U^+$ of two open connected sets $U^-$ and $U^+$ in the line $L$. Since $U^-\cup U^+\subseteq L\setminus F= \bigcup_{s\in\{-1,1\}}\bigcup_{a\in A^s}I^s_a$, the connectedness of $U^-,U_+$ ensures that $U^-\subseteq I^s_a\subseteq  F_a^{s}$ and $U^+\subseteq I^t_b\subseteq F_b^t$ for some signs $s,t\in \{-1,+1\}$ and points $a\in A^s$ and $b\in A^t$.

 We claim that $s\ne t$. To derive a contradiction, assume that $s=t$ and observe that $u\in \overline{U^+}\cap \overline{U^-}\subseteq F^s_a\cap F^s_b$, which implies $a=b$ and $u\in F^s_a$, according to Claim~\ref{cl:2}. Then $U^-\cup\{u\}\cup U^+$ is an open neighborhood of $u$ in $L\cap F^s_a$, which implies $u\in I^s_a$ and contradicts $u\in F$. This contradiction shows that $s\ne t$. In this case we lose no generality assuming that $s=-1$ and $t=+1$ (otherwise exchange the sets $U^-$ and $U^+$ by their places around the point $u$ in the line $L$). 

By Claim~\ref{cl:2a}, $u\in \overline{U^-}\cap\overline{U^+}\subseteq F^{-1}_a\cap F^{+1}_b=\{\frac{a-b}2\}$ and hence $2{\cdot}u=a-b\in H-H=H$. Since the group $H$ intersects every Cantor set in $L$, the intersection $H\cap L$ is dense in $L$. Then we can choose a point $x\in H\cap U^-\subseteq F^{-1}_a$ so close to $u$ that the point $y\defeq 2u-x\in H-H=H$ belongs to the interval $U^+\subseteq F^{+1}_b$. Then $x,y$ are two distinct points of the group $H$ such that $f_\alpha(y)=y+b=2{\cdot}u-x+b=(a-b)-x+b=-x+a=f_\alpha(x)$, which contradicts the bijectivity of the function $f=f_\alpha{\restriction}_H$. This contradiction shows that the set $F$ is empty and hence
$L\subseteq \bigcup_{s\in\{-1,+1\}}\bigcup_{a\in A^s}I_a^s$.  Now the connectedness of the line $L$ ensures that $L\subseteq I^s_a\subseteq F^s_a$ for some $s\in\{-1,+1\}$ and $a\in A^s\subseteq H_{<\alpha}$.
\end{proof}

By Claim~\ref{cl:3}, for every  line $\ell$ in the normed space $K$, there exist a sign $s\ell\in\{-1,+1\}$ and a point $a\ell\in H_{<\alpha}$ and such that $L\subseteq F_{a\ell}^{s\ell}$.

\begin{claim} \label{cl:4}For any intersecting lines $\ell_1,\ell_2$ in $K$, we have $s\ell_1=s\ell_2$ and $a\ell_1=a\ell_2$.
\end{claim}

\begin{proof} To derive a contradiction, assume that $s\ell_1\ne s\ell_2$. Let $o$ be the unique common point of the lines $\ell_1,\ell_2$. Then $\{o\}=\ell_1\cap \ell_2\subseteq F^{s\ell_1}_{a\ell_1}\cap F^{s\ell_2}_{a\ell_2}$ and hence $F^{s\ell_1}_{a\ell_1}\cap F^{s\ell_2}_{a\ell_2}=\{o\}$, by Claim~\ref{cl:2a}. Choose any points $x_1\in \ell_1\setminus\{o\}$ and $x_2\in\ell_2\setminus\{o\}$, and consider the line $\ell\defeq \{(1-t){\cdot}x_1+t\cdot x_2:t\in\IR\}$. 
Since $s\ell\in\{-1,+1\}=\{s\ell_1,s\ell_2\}$, there exists a number $i\in\{1,2\}$ such that $s\ell=s\ell_i$. Let $j$ be the unique number in the set $\{1,2\}\setminus\{i\}$. By Claim~\ref{cl:2}, $x_i\in\ell\cap \ell_i\subseteq F^{s\ell}_{a\ell}\cap F^{s\ell_i}_{a\ell_i}$ implies $a\ell=a\ell_i$. Then $x_j\in \ell\cap \ell_j\subseteq F^{s\ell}_{a\ell}\cap F^{s\ell_j}_{a\ell_j}=F^{s\ell_i}_{a\ell_i}\cap F^{s\ell_j}_{a\ell_j}=\{o\}$, which contradicts the choice of the point $x_j\ne o$. This contradiction shows that $s\ell_1=s\ell_2$. Since $\varnothing \ne \ell_1\cap \ell_2\subseteq F^{s\ell_1}_{a\ell_1}\cap F^{s\ell_2}_{a\ell_2}$, we can apply Claim~\ref{cl:2} and conlcude that $a\ell_1=a\ell_2$.
\end{proof}

\begin{claim}\label{cl:5} For any lines $\ell_1$ and $\ell_2$ in $K$, we have $s\ell_1=s\ell_2$ and $a\ell_1=a\ell_2$.
\end{claim}

\begin{proof} Given any lines $\ell_1,\ell_2$ in $K$, choose a line $\ell$ in $K$ such that $\ell_1\cap \ell\ne\emptyset\ne \ell_2\cap \ell$. Claim~\ref{cl:2} ensures that $s\ell_1=s\ell=s\ell_2$ and $a\ell_1=a\ell=a\ell_2$.
\end{proof}

By Claim~\ref{cl:5}, there exist $s\in\{-1,+1\}$ and $a\in H_{<\alpha}$ such that $s\ell=s$ and $a\ell=a$ for all lines $\ell$ in $K$. Since the normed space $K$ contains more than one point, every point $x\in K$ is contained in some line $\ell_x\subseteq K$. Then $x\in \ell_x\subseteq F^{s\ell_x}_{a\ell_x}=F^s_a$ and hence $K\subseteq F^s_a$, which implies $f(x)=f_\alpha(x)=s{\cdot}x+a$ for all $x\in H$, and witnesses that the group $H$ is plastically rigid. \hfill $\square$
		
\section{Final remarks and open problems}\label{s:openprob}

By Theorem~\ref{t:main1}, no countable dense subspace of the real line is plastic. On the other hand, by Example~\ref{e:Bielas}, the real line contains a plastic dense subspace $X$ of cardinality $|X|=\aleph_1$. However, the plastic subspace constructed in Example~\ref{e:Bielas} is not a subgroup of the real line. 

\begin{problem} Is it true that every dense subgroup $G\subset\IR$ of cardinality $|G|<\mathfrak c$ is not plastic?
\end{problem}

A subset $X$ of a real line $\IR$ is called {\em $\aleph_1$-dense} if for every nonempty open set $U\subseteq\IR$, the intersection $X\cap U$ has cardinality $\aleph_1$. By \cite[Theorem 9.6]{Baum}, under PFA (the Proper Forcing Axiom), any two $\aleph_1$-dense subspaces of the real line are order isomorphic and hence homeomorphic. It is known \cite{Velickovic} that PFA implies that $\mathfrak c=\aleph_2$.

\begin{problem} Is it true that under {\rm PFA} every $\aleph_1$-dense subspace (subgroup) of the real line is not plastic?
\end{problem}


\begin{remark}
The group $H$ constructed in the proof of Theorem~\ref{t:main3} has cardinality $\mathfrak c$ (because it intersects every Cantor set in $K$) and is not analytic. Indeed, assuming that $H$ is analytic, we can apply Proposition~\ref{p:Cantor-set} and find a Cantor set $C\subseteq H$. Then $H$ would not intersect the Cantor sets $x+C$ for all $x\in K\setminus H$, which contradicts the property (3) of the inductive construction of the group $H$. The non-analycity of the plastic subgroup $H$ motivates the following problem.
\end{remark}

\begin{problem} Is it true that every normed space contains a plastic dense analytic subgroup? In particular, is there a plastic dense analytic subgroup in the real line?
\end{problem}

The dense plastic space $X$, constructed in the proof of Theorem~\ref{t:main2} is also not analytic, at least in case $Y=\IR$. However the real line does contain an analytic (even $\sigma$-compact) plastic dense subspace.

\begin{example}\label{ex:main} The real line contains a plastic dense $\sigma$-compact $k$-crowded subspace.
\end{example}
	
\begin{proof} Consider the closed set  $Z\defeq \bigcup_{n\in\IZ}[2n,2n+1]$ in $\IR$. Fix any countable base $\{B_n\}_{n\in\w}$ of the topology of the space $\IR\setminus Z$ such that each open set $B_n$ is not empty and hence contains a Cantor set $C_n$. Then the union $C\defeq \bigcup_{n\in\w}C_n$ is a dense $\sigma$-compact $k$-crowded set in $\IR\setminus Z$. Moreover, $C$ is zero-dimensional, being the countable union of compact zero-dimensional spaces, see The Countable Sum Theorem~7.2.1 in \cite{Engelking}.  Then $X:=C\cup Z$ is a dense $\sigma$-compact $k$-crowded subspace of the real line. We claim that the metric space $X$ is plastic.  Given any non-expansive bijection $f\colon X \to X$, we should prove that $f$ is an isometry.

\begin{claim}\label{cl:7} $f[\IZ]\subseteq \IZ$.
\end{claim}

\begin{proof} Given any integer number $z\in\IZ$, find a unique number $m\in\IZ$ such that $z\in\{2m,2m+1\}$. 

Assuming that $f(z)\notin Z$, we can find, by the continuity of $f$, a connected neighborhood $U$ of $z$ in $[2m,2m+1]$ such that $f[U]\subseteq X\setminus Z\subseteq C$.  Then the connected subspace $f[U]$ of the zero-dimensional space $C$ is a singleton, which contradicts the bijectivity of $f$. This contradiction shows that $f(z)\in Z$.

Next, assume that $f(z)\in Z\setminus \IZ=\bigcup_{n\in\IZ}(2n,2n+1)$ and hence $f(z)\in W\defeq (2k,2k+1)$ for some $k\in\IZ$. 
The continuity and bijectivity of $f$ ensures that for every $n\in\IZ$, the set $I_n\defeq\{f(x):2n\le x\le 2n+1\}$ is a closed interval in the real line and the set $J_n\defeq\{f(x):2n<x<2n+1\}$ is its interior in $\IR$. 
Consider the closed subset $F=W\setminus\bigcup_{n\in\IZ}J_n$ in the open interval $W=(2k,2k+1)$, and observe that it contains the countable set $B\defeq W\cap f[\IZ]\ni f(z)$. 
It follows that $F\subseteq B\cup f[C]$. The continuity and bijectivity of $f$ implies that the image $f[C]$ of the $\sigma$-compact zero-dimensional space $C$ is $\sigma$-compact and zero-dimensional.  Then the set $F\subseteq B\cup f[C]$ is zero-dimensional and nowhere dense in $W$, which implies that the complement $W\setminus F=W\cap\bigcup_{n\in\IZ}J_n$ is dense in $W$.

We claim that the set $B$ is dense in $F$. Indeed, take any point $x\in F$ and any $\e>0$. By the density of the set $W\cap\bigcup_{n\in\IZ}J_n$ in $W$, there exist $n\in \IZ$ and point $y\in (0,1)\cap J_n$ such that $|x-y|<\e$. Since $x\notin J_n$, the interval $[x,y]$ contains a boundary point $b\in B$ of the open interval $J_n$. Then $|x-b|\le |x-y|<\e$, witnessing that $B$ is dense in $F$. The continuity and bijectivity of $f$ implies also that the boundary point $b$ of $J_n$ is not isolated in the set $F$. 

 By the bijectivity of $f$, the $\sigma$-compact set $W\cap f[C]$ does not intersect the dense subset $B=W\cap f[\IZ]$ of $W$ and hence is of the first Baire category in $W$. Then the Polish space $W=B\cup (W\cap f[C])$ is of the first Baire category, which contradicts the Baire Theorem. This contraduction shows that $f(z)\in \IZ$.
 \end{proof}
 
By Claim~\ref{cl:7}, $f[\IZ]\subseteq\IZ$. Consider the integer number $b\defeq f(0)$. Since $0<|f(0)-f(1)|\le 1$, there exists a number $s\in\{-1,1\}$ such that $f(1)=b+s$. Consider the isometry $g:\IR\to\IR$, $g:x\mapsto b+s{\cdot}x$, of the real line. 

\begin{claim} $f(z)=g(x)$ for all $z\in\IZ$.
\end{claim}

\begin{proof} Taking into account that $0<|f(-1)-f(0)|\le 1$, we conclude that $f(-1)=b-s=g(-1)$ is a unique integer number such that $|f(-1)-f(0)|=1$ and $f(-1)\ne f(1)=b+s$.
Therefore, the equality $f(z)=g(z)$ holds for all integer numbers $z$ with $|z|\le 1$.

 Assume that for some integer $n\ge 2$, the equality $f(z)=g(z)$ has been proved for all integer numbers $z\in Z$ with $|z|<n$. The inductive hypothesis implies $f(n-2)=g(n-2)=b+s\cdot(n-2)$ and $f(n-1)=g(n-1)=b+s\cdot(n-1)$. Taking into account that $0<|f(n)-f(n-1)|\le 1$, and $f(n)\ne f(n-2)=f(n-1)-s$, we conclude that $f(n)=f(n-1)+s=b+s\cdot n=g(n)$. By analogy we can show that $f(-n)=f(-n+1)-s=b+s\cdot(-n)=g(-n)$.
This completes the inductive step.

The Principle of Mathematical Induction ensures that $f(z)=g(z)$ for all $z\in\IZ$.
\end{proof}

\begin{claim}\label{cl:9} $f(x)=g(x)$ for all $x\in X$.
\end{claim}

\begin{proof} Given any $x\in X$, find an integer number $n$ such that $x\in[n,n+1]$. Observe that $[n,n+1]\cap X\subseteq\IR$ is a metric interval with endpoints $n$ and $n+1$.
Since $f,g$ are two non-expansive maps with $f(n)=g(n)$ and $f(n+1)=g(n+1)$, we can apply Lemma~\ref{l:4} and conclude that $f(x)=g(x)$.
\end{proof}

By Claim~\ref{cl:9}, the non-expansive bijection $f=g{\restriction}_X$ is an isometry, witnessing that the $\sigma$-compact space $X$ is plastic.
\end{proof}

Example~\ref{ex:main} and Theorems~\ref{t:main2}, \ref{t:main3} motivate the following problems.

\begin{problem}
Is it true that every $k$-crowded separable metric space contains a plastic dense $\sigma$-compact $k$-crowded subspace?
\end{problem}

\begin{problem}
Is it true that every strictly convex metric abelian group contains a dense plastic $\sigma$-compact subgroup? In particular, does $\IR$ contain such a subgroup?
\end{problem}

Since the real plane $\IR\times\IR$ endowed with the Euclidean norm $\|(x,y)\|_2=\sqrt{x^2+y^2}$ is a strictly convex normed space, it contains a plastic dense subgroup, by Theorem~\ref{t:main3}. On the other hand, for the plane $\IR\times\IR$ endowed with the $\ell_1$-norm $\|(x,y)\|_1=|x|+|y|$, Theorem~\ref{t:main3} is not applicable, which leads to the following open problem.

\begin{problem} Is there a dense plastic subgroup in the real plane with the $\ell_1$-norm?
\end{problem}

\section{Acknowledgement}

The authors express their sincere thanks to Wojciech Bielas for the solution a problem from the initial version of our paper that leads to construction of a plastic space in Example~\ref{e:Bielas} and characterization of the Continuum Hypothesis in Corollary~\ref{c:Bielas-CH}.

\newpage
	
	

\end{document}